\documentclass[12pt]{amsart}

\usepackage{mathrsfs,amssymb}

\newtheorem{theorem}{Theorem}[section]
\newtheorem{lemma}[theorem]{Lemma}
\newtheorem{prop}[theorem]{Proposition}
\newtheorem{cor}[theorem]{Corollary}

\theoremstyle{definition}

\theoremstyle{remark}

\numberwithin{equation}{section}

\let \la=\lambda
\let \e=\varepsilon
\let \d=\delta
\let \o=\omega
\let \a=\alpha
\let \f=\varphi
\let \b=\beta

\let \O=\Omega

\let \G=\Gamma
\let \ga=\gamma

\begin{document}

\title[sharp aperture-weighted estimates]
{On sharp aperture-weighted estimates for square functions}

\author{Andrei K. Lerner}
\address{Department of Mathematics,
Bar-Ilan University, 52900 Ramat Gan, Israel}
\email{aklerner@netvision.net.il}

\begin{abstract}
Let $S_{\a,\psi}(f)$ be the square function defined by means of the cone in ${\mathbb R}^{n+1}_{+}$ of aperture $\a$, and a standard kernel $\psi$.
Let $[w]_{A_p}$ denote the $A_p$ characteristic of the weight $w$. We show that for any $1<p<\infty$ and $\a\ge 1$,
$$\|S_{\a,\psi}\|_{L^p(w)}\lesssim \a^n[w]_{A_p}^{\max(\frac{1}{2},\frac{1}{p-1})}.$$
For each fixed $\a$ the dependence on $[w]_{A_p}$ is sharp. Also, on all class $A_p$ the result is sharp in $\a$. 
Previously this estimate was proved in the case $\a=1$ using the intrinsic
square function. However, that approach does not allow to get the above estimate with sharp dependence on $\a$.
Hence we give a different proof suitable for all $\a\ge 1$ and  avoiding the notion of the intrinsic square function.
\end{abstract}

\keywords{Littlewood-Paley operators, sharp
weighted inequalities, sharp aperture dependence.}

\subjclass[2010]{42B20,42B25}

\maketitle

\section{Introduction}
Let $\psi$ be an integrable function, $\int_{{\mathbb R}^n}\psi=0$, and, for some $\e>0$,
\begin{equation}\label{cond}
|\psi(x)|\le \frac{c}{(1+|x|)^{n+\e}}\quad\text{and}\quad\int_{{\mathbb R}^n}|\psi(x+h)-\psi(x)|dx\le c|h|^{\e}.
\end{equation}

Let ${\mathbb R}^{n+1}_+={\mathbb R}^{n}\times{\mathbb R}_{+}$ and
$\G_{\a}(x)=\{(y,t)\in {\mathbb{R}}^{n+1}_+:|y-x|<\a t\}$.  Set
$\psi_t(x)=t^{-n}\psi(x/t)$. Define the square function $S_{\a,\psi}(f)$ by
$$
S_{\a,\psi}(f)(x)=\left(\int_{\G_{\a}(x)}|f*\psi_t(y)|^2\frac{dydt}{t^{n+1}}\right)^{1/2}\quad(\a>0).
$$
We drop the subscript $\a$ if $\a=1$.

Given a weight $w$, define its $A_p$ characteristic by
$$[w]_{A_p}=\sup_{Q}\left(\frac{1}{|Q|}\int_Qw\,dx\right)
\left(\frac{1}{|Q|}\int_Qw^{-\frac{1}{p-1}}\,dx\right)^{p-1},$$
where the supremum is taken over all cubes $Q\subset {\mathbb R}^n$.

It was proved in \cite{Lerner} that for any $1<p<\infty$,
\begin{equation}\label{sqsh}
\|S_{\psi}\|_{L^p(w)}\le c_{p,n,\psi}[w]_{A_p}^{\max(\frac{1}{2},\frac{1}{p-1})},
\end{equation}
and this estimate is sharp in terms of $[w]_{A_p}$ (we also refer to \cite{Lerner} for a detailed history of
closely related results).

Similarly one can show that
\begin{equation}\label{sqsh1}
\|S_{\a,\psi}\|_{L^p(w)}\le c_{p,n,\psi}\ga(\a)[w]_{A_p}^{\max(\frac{1}{2},\frac{1}{p-1})}\quad(\a\ge 1,1<p<\infty);
\end{equation}
however, the sharp dependence on $\a$ in this estimate cannot be determined by means of the approach from \cite{Lerner}.
The aim of this paper is to find the sharp $\ga(\a)$ in (\ref{sqsh1}).

Let us explain first why the method from \cite{Lerner} gives a rough estimate for $\ga(\a)$.
The proof in \cite{Lerner} was based on the intrinsic square function $G_{\a,\b}(f)$ by M. Wilson \cite{Wil} defined as follows.
For $0<\b\le 1$, let ${\mathcal C}_{\b}$ be the family of functions
supported in the unit ball with mean zero and such
that for all $x$ and $x'$, $|\f(x)-\f(x')|\le |x-x'|^{\b}$. If $f\in
L^1_{\text{loc}}({\mathbb R}^n)$ and $(y,t)\in {\mathbb R}^{n+1}_+$,
we define
$
A_{\b}(f)(y,t)=\sup_{\f\in {\mathcal C}_{\b}}|f*\f_t(y)|
$
and
$$
G_{\a,\b}(f)(x)=\left(\int_{\G_{\a}(x)}
\big(A_{\b}(f)(y,t)\big)^2\frac{dydt}{t^{n+1}}\right)^{1/2}.
$$
Set $G_{1,\b}(f)=G_{\b}(f)$.

The intrinsic square function has several interesting features (established in \cite{Wil}).
First, though $G_{\b}(f)$ is defined by means of kernels with uniform compact support, it
pointwise dominates $S_{\psi}(f)$. Also there is a
pointwise relation between $G_{\a,\b}(f)$ with different apertures:
\begin{equation}\label{ap}
G_{\a,\b}(f)(x)\le \a^{(3/2)n+\b}G_{\b}(f)(x)\quad(\a\ge 1).
\end{equation}
Notice that for the usual square functions $S_{\a,\psi}(f)$ such a pointwise relation is not available.

In \cite{Lerner}, (\ref{sqsh}) with $G_{\b}(f)$ instead
of $S_{\psi}(f)$ was obtained. Combining this with (\ref{ap}), we would obtain that one can take $\ga(\a)=\a^{(3/2)n+\b}$ in (\ref{sqsh1})
assuming that $\psi$ is compactly supported. For non-compactly supported $\psi$ some additional ideas from \cite{Wil} can be used that
lead to even worst estimate on $\ga(\a)$. Observe also that it is not clear to us whether (\ref{ap}) can be improved.

It is easy to see that the dependence $\ga(\a)=\a^{(3/2)n+\b}$ in (\ref{sqsh1}) is far from the sharp one. For instance, it is obvious that the
information on $\b$ should not appear in (\ref{sqsh1}). All this indicates that the intrinsic square function approach is not suitable for our
purposes in determining the sharp $\ga(\a)$.

Suppose we seek for $\ga(\a)$ in the form $\ga(\a)=\a^{r}$. Then a simple observation shows that
$r\ge n$ for any $1<p<\infty$. Indeed, consider the Littlewood-Paley function $g_{\mu,\psi}^*(f)$ defined by
$$
g_{\mu,\psi}^*(f)(x)=\biggl(\iint_{{\mathbb
R}^{n+1}_+}\biggl(\frac{t}{t+|x-y|}\biggr)^{\mu
n}|f*\psi_t(y)|^2\frac{dydt}{t^{n+1}}\biggr)^{1/2}.
$$
Using the standard estimate
$$
g_{\mu,\psi}^*(f)(x)\le S_{\psi}(f)(x)+\sum_{k=0}^{\infty}2^{-k\mu n/2}S_{2^{k+1},\psi}(f)(x),
$$
we obtain that (\ref{sqsh1}) for some $p=p_0$ and $\ga(\a)=\a^{r_0}$ implies
\begin{equation}\label{stan}
\|g_{\mu,\psi}^*\|_{L^{p_0}(w)}\lesssim \Big(\sum_{k=0}^{\infty}2^{-k\mu n/2}2^{kr_0}\Big)
[w]_{A_{p_0}}^{\max(\frac{1}{2},\frac{1}{p_0-1})}.
\end{equation}
This means that if $\mu>2r_0/n$, then $g_{\mu,\psi}^*$ is bounded on $L^{p_0}(w)$, $w\in~A_{p_0}$. From this, by the Rubio de Francia extrapolation
theorem, $g_{\mu,\psi}^*$ is bounded on the unweighted $L^p$ for any $p>1$, whenever $\mu>2r_0/n$. But it is well known \cite{F} that
$g_{\mu,\psi}^*$ is not bounded on $L^p$ if $1<\mu<2$ and  $1<p\le 2/\mu$. Hence, if $r_0<n$, we would obtain a contradiction to the latter fact for $p$ sufficiently
close to $1$.

Our main result shows that for any $1<p<\infty$ one can take the optimal power growth $\ga(\a)=\a^n$.

\begin{theorem}\label{mainr} For any $1<p<\infty$ and for all $1\le \a<\infty$,
$$
\|S_{\a,\psi}\|_{L^p(w)}\le c_{p,n,\psi}\a^n[w]_{A_p}^{\max(\frac{1}{2},\frac{1}{p-1})}.
$$
\end{theorem}

By (\ref{stan}), we immediately obtain the following.

\begin{cor}\label{corol} Let $\mu>2$. Then for any $1<p<\infty$,
$$\|g_{\mu,\psi}^*(f)\|_{L^{p}(w)}\le c_{p,n,\mu,\psi}[w]_{A_p}^{\max(\frac{1}{2},\frac{1}{p-1})}.$$
\end{cor}

Observe that if $\mu=2$, then $g_{2,\psi}^*$ is also bounded on $L^p(w)$ for $w\in A_p$ (see \cite{MW}).
However, the sharp dependence on $[w]_{A_p}$ in the corresponding $L^p(w)$ inequality is unknown to us.

We emphasize that the growth $\ga(\a)=\a^n$ is best possible in the weighted $L^p(w)$ estimate for $w\in A_p$.
In the unweighted case a better dependence on $\a$ is known, namely, $\|S_{\a,\psi}\|_{L^p}\le c_{p,n,\psi}\a^{\frac{n}{\min(p,2)}},$
see \cite{A,T}.

Some words about the proof of Theorem \ref{mainr}. As in \cite{Lerner}, we use here the local mean oscillation decomposition.
But in \cite{Lerner} we worked with the intrinsic square function, and due to the fact that this operator is defined by uniform
compactly supported kernels, we arrived to the operator
$${\mathcal A}(f)(x)=\Big(\sum_{j,k}(f_{\ga Q_j^k})^2\chi_{Q_j^k}(x)\Big)^{1/2},$$
where $Q_j^k$ is a sparse family and $\ga>1$. This operator can be handled sufficiently easy.

Here we work with the square function $S_{\a,\psi}(f)$ directly, more precisely we consider its smooth variant
$\widetilde S_{\a,\psi}(f)$. Applying the local mean oscillation decomposition to $\widetilde S_{\a,\psi}(f)$, we obtain that $S_{\a,\psi}(f)$ is
essentially pointwise bounded by $\a^n{\mathcal B}(f)$, where
$${\mathcal B}(f)(x)=\sum_{m=0}^{\infty}\frac{1}{2^{m\d}}\Big(\sum_{j,k}(f_{2^m Q_j^k})^2\chi_{Q_j^k}(x)\Big)^{1/2}\quad(\d>0).$$
Observe that this pointwise aperture estimate is interesting in its own right. In order to handle ${\mathcal B}$, we use a mixture of ideas from
recent papers on a simple proof of the $A_2$ conjecture \cite{Ler} and sharp weighted estimates for multilinear Calder\'on-Zygmund operators \cite{DLP}.
In particular, similarly to \cite{Ler}, we obtain the $X^{(2)}$-norm boundedness of ${\mathcal B}$ by ${\mathcal A}$ on an arbitrary Banach function space $X$.

The paper is organized as follows. Next section contains some preliminary information. In Section 3, we obtain the main estimate, namely, the local mean oscillation
estimate of $\widetilde S_{\a,\psi}(f)$. The proof of Theorem~\ref{mainr} is contained in Section 4. Section 5 contains some concluding remarks concerning the sharp aperture-weighted weak type estimates for $S_{\a,\psi}(f)$.

\section{Preliminaries}
\subsection{A weak type $(1,1)$ estimate for square functions}
It is well known that the operator $S_{\a,\psi}$ is of weak type $(1,1)$. However, we could not
find in the literature the sharp dependence on $\a$ in the corresponding inequality. Hence we give
below an argument based on general square functions.

For a measurable function $F$ on ${\mathbb R}^{n+1}_+$ define
$$
S_{\a}(F)(x)=\biggl(\int_{\G_{\a}(x)}|F(y,t)|^2\frac{dydt}{t^{n+1}}\biggr)^{1/2}.
$$

\begin{lemma}\label{weak} For any $\a\ge 1$,
\begin{equation}\label{angle}
\|S_{\a}(F)\|_{L^{1,\infty}}\le c_n\a^n\|S_{1}(F)\|_{L^{1,\infty}}.
\end{equation}
\end{lemma}

\begin{proof}
We will use the following estimate which can be found in \cite[p.~315]{T}:
if $\O\subset {\mathbb R}^n$ is an open set and $U=\{x\in {\mathbb R}^n:M\chi_{\O}(x)>1/2\a^n\},$
where $M$ is the Hardy-Littlewood maximal operator, then
$$\int_{{\mathbb R}^n\setminus U}S_{\a}(F)(x)^2dx\le 2\a^n\int_{{\mathbb R}^n\setminus \O}S_{1}(F)(x)^2dx$$
(observe that the definitions of $S_{\a}(F)$ here and in \cite{T} are differ by the factor $\a^{n/2}$.)

Let $\O_{\xi}=\{x:S_{1}(F)(x)>\xi\}$. Using the weak type $(1,1)$ of $M$, Chebyshev's inequality and the above estimate, we obtain
\begin{eqnarray*}
&&|\{x\in {\mathbb R}^n:S_{\a}(F)(x)>\xi\}|\\
&&\le |U_{\xi}|+|\{x\in {\mathbb R}^n\setminus U_{\xi}:S_{\a}(F)(x)>\xi\}|\\
&&\le c_n\a^n|\{x:S_{1}(F)(x)>\xi\}|+\frac{1}{\xi^2}\int_{{\mathbb R}^n\setminus U_{\xi}}S_{\a}(F)(x)^2dx\\
&&\le c_n\a^n|\{x:S_{1}(F)(x)>\xi\}|+\frac{2\a^n}{\xi^2}\int_{{\mathbb R}^n\setminus \O_{\xi}}S_{1}(F)(x)^2dx.
\end{eqnarray*}
Further,
$$
\int_{{\mathbb R}^n\setminus \O_{\xi}}S_{1}(F)(x)^2dx\le 2\int_{0}^{\xi}\la|\{x:S_{1}(F)(x)>\la\}|d\la
\le 2\xi\|S_{1}(F)\|_{L^{1,\infty}}.
$$
Combining this with the previous estimate gives
$$|\{x:S_{\a}(F)(x)>\xi\}|\le c_n\a^n|\{x:S_{1}(F)(x)>\xi\}|+\frac{4\a^n}{\xi}\|S_{1}(F)\|_{L^{1,\infty}},$$
which proves (\ref{angle}).
\end{proof}

Note that the sharp strong $L^p$ estimates related square functions of different apertures were obtained recently in \cite{A}.

By Lemma \ref{weak} and by the weak type $(1,1)$ of $S_{\psi}(f)$ \cite{GCR},
\begin{equation}\label{weakt}
\|S_{\a,\psi}(f)\|_{L^{1,\infty}}\le c_{n,\psi}\a^{n}\|f\|_{L^1}.
\end{equation}

\subsection{Dyadic grids and sparse families}
Recall that the standard dyadic grid in ${\mathbb R}^n$ consists of the cubes
$$2^{-k}([0,1)^n+j),\quad k\in{\mathbb Z}, j\in{\mathbb Z}^n.$$
Denote the standard grid by ${\mathcal D}$.

By a {\it general dyadic grid} ${\mathscr{D}}$ we mean a collection of
cubes with the following properties: (i)
for any $Q\in {\mathscr{D}}$ its sidelength $\ell_Q$ is of the form
$2^k, k\in {\mathbb Z}$; (ii) $Q\cap R\in\{Q,R,\emptyset\}$ for any $Q,R\in {\mathscr{D}}$;
(iii) the cubes of a fixed sidelength $2^k$ form a partition of ${\mathbb
R}^n$.

Given a cube $Q_0$, denote by ${\mathcal D}(Q_0)$ the set of all
dyadic cubes with respect to $Q_0$, that is, the cubes from ${\mathcal D}(Q_0)$ are formed
by repeated subdivision of $Q_0$ and each of its descendants into $2^n$ congruent subcubes.
Observe that if $Q_0\in {\mathscr{D}}$, then each cube from ${\mathcal D}(Q_0)$ will also
belong to ${\mathscr{D}}$.

We will use the following proposition from \cite{HP}.

\begin{prop}\label{prhp} There are $2^n$ dyadic grids ${\mathscr{D}}_{i}$ such that for any cube $Q\subset {\mathbb R}^n$ there exists a cube $Q_{i}\in {\mathscr{D}}_{i}$
such that $Q\subset Q_{i}$ and $\ell_{Q_{i}}\le 6\ell_Q$.
\end{prop}

We say that $\{Q_j^k\}$ is a {\it sparse family} of cubes if:
(i)~the cubes $Q_j^k$ are disjoint in $j$, with $k$ fixed;
(ii) if $\Omega_k=\cup_jQ_j^k$, then $\Omega_{k+1}\subset~\Omega_k$;
(iii) $|\Omega_{k+1}\cap Q_j^k|\le \frac{1}{2}|Q_j^k|$.

\subsection{A ``local mean oscillation decomposition"}
The non-increasing rearrangement of a measurable function $f$ on ${\mathbb R}^n$ is defined by
$$f^*(t)=\inf\{\a>0:|\{x\in {\mathbb R}^n:|f(x)|<\a\}|<t\}\quad(0<t<\infty).$$

Given a measurable function $f$ on ${\mathbb R}^n$ and a cube $Q$,
the local mean oscillation of $f$ on $Q$ is defined by
$$\o_{\la}(f;Q)=\inf_{c\in {\mathbb R}}
\big((f-c)\chi_{Q}\big)^*\big(\la|Q|\big)\quad(0<\la<1).$$

By a median value of $f$ over $Q$ we mean a possibly nonunique, real
number $m_f(Q)$ such that
$$\max\big(|\{x\in Q: f(x)>m_f(Q)\}|,|\{x\in Q: f(x)<m_f(Q)\}|\big)\le |Q|/2.$$

It is easy to see that the set of all median values of $f$ is either one point or the closed interval. In the latter case we will assume for
the definiteness that $m_f(Q)$ is the {\it maximal} median value. Observe that it follows from the definitions that
\begin{equation}\label{pro1}
|m_f(Q)|\le (f\chi_Q)^*(|Q|/2).
\end{equation}

Given a cube $Q_0$, the dyadic local sharp maximal
function $m^{\#,d}_{\la;Q_0}f$ is defined by
$$m^{\#,d}_{\la;Q_0}f(x)=\sup_{x\in Q'\in
{\mathcal D}(Q_0)}\o_{\la}(f;Q').$$

The following theorem was proved in \cite{Ler1} (its very similar version can be found in \cite{L1}).

\begin{theorem}\label{decom1} Let $f$ be a measurable function on
${\mathbb R}^n$ and let $Q_0$ be a fixed cube. Then there exists a
(possibly empty) sparse family of cubes $Q_j^k\in {\mathcal D}(Q_0)$
such that for a.e. $x\in Q_0$,
$$
|f(x)-m_f(Q_0)|\le
4m_{\frac{1}{2^{n+2}};Q_0}^{\#,d}f(x)+2\sum_{k,j}
\o_{\frac{1}{2^{n+2}}}(f;Q_j^k)\chi_{Q_j^k}(x).
$$
\end{theorem}

\section{A key estimate}
In this section we will obtain the main local mean oscillation estimate of $S_{\a,\psi}$.
We consider a smooth version of $S_{\a,\psi}$ defined as follows. Let $\Phi$ be a Schwartz function such that
$$\chi_{B(0,1)}(x)\le \Phi(x)\le \chi_{B(0,2)}(x).$$
Define
$$
\widetilde S_{\a,\psi}(f)(x)=\left(\iint_{{\mathbb R}^{n+1}_+}\Phi\Big(\frac{x-y}{t\a}\Big)|f*\psi_t(y)|^2\frac{dydt}{t^{n+1}}\right)^{1/2}\quad(\a>0).
$$
It is easy to see that
$$
S_{\a,\psi}(f)(x)\le \widetilde S_{\a,\psi}(f)(x)\le
S_{2\a,\psi}(f)(x).$$
Hence, by (\ref{weakt}),
\begin{equation}\label{weak1}
\|\widetilde S_{\a,\psi}(f)\|_{L^{1,\infty}}\le c_{n,\psi}\a^{n}\|f\|_{L^1}.
\end{equation}

\begin{lemma}\label{oscg} For any cube $Q\subset {\mathbb R}^n$,
\begin{equation}\label{olan}
\o_{\la}(\widetilde S_{\a,\psi}(f)^2;Q)\le
c_{n,\la,\psi}\a^{2n}\sum_{k=0}^{\infty}\frac{1}{2^{k\d}}\left(\frac{1}{|2^kQ|}\int_{2^kQ}|f|\right)^2,
\end{equation}
where $\d=\e$ from condition (\ref{cond}) if $\e<1$, and $\d<1$ if $\e=1$.
\end{lemma}

\begin{proof} Given a cube $Q$, let $T(Q)=\{(y,t):y\in Q,
0<t<{\ell}_Q\},$ where ${\ell}_Q$ denotes the side length of $Q$.
For $x\in Q$ we decompose
$\widetilde S_{\a,\psi}(f)(x)^2$ into the sum of
$$
I_1(f)(x)=\iint_{T(2Q)}
\Phi\Big(\frac{x-y}{t\a}\Big)|f*\psi_t(y)|^2\frac{dydt}{t^{n+1}}
$$
and
$$
I_2(f)(x)=\iint_{{\mathbb R}^{n+1}_+\setminus T(2Q)
}\Phi\Big(\frac{x-y}{t\a}\Big)|f*\psi_t(y)|^2\frac{dydt}{t^{n+1}}.
$$

Let us show first that
\begin{equation}\label{1cl}
(I_1(f)\chi_Q)^*(\la|Q|)\le c_{n,\la,\psi}\a^{2n}\sum_{k=0}^{\infty}\frac{1}{2^{k\e}}\left(\frac{1}{|2^kQ|}\int_{2^kQ}|f|\right)^2.
\end{equation}
Using that $(a+b)^2\le 2(a^2+b^2)$, we get
$$
I_1(f)(x)\le 2\big(I_1(f\chi_{4Q})(x)+I_1(f\chi_{{\mathbb R}^n\setminus 4Q})(x)\big).
$$
Hence,
\begin{eqnarray}
(I_1(f)\chi_Q)^*(\la|Q|)&\le&2\big((I_1(f\chi_{4Q}))^*(\la|Q|/2)\label{reares}\\
&+&(I_1(f\chi_{{\mathbb R}^n\setminus 4Q})\chi_Q)^*(\la|Q|/2)\big).\nonumber
\end{eqnarray}

By (\ref{weak1}),
\begin{eqnarray}
(I_1(f\chi_{4Q}))^*(\la|Q|/2)&\le& (\widetilde S_{\a,\psi}(f\chi_{4Q}))^*(\la|Q|/2)^2\label{weak11}\\
&\le& c_{n,\la,\psi}\a^{2n}\left(\frac{1}{|4Q|}\int_{4Q}|f|\right)^2.\nonumber
\end{eqnarray}
Further, by (\ref{cond}), for $(y,t)\in T(2Q)$,
\begin{eqnarray*}
|(f\chi_{{\mathbb R}^n\setminus 4Q})*\psi_t(y)|&\le& c_{\psi}t^{\e}\int_{{\mathbb R}^n\setminus 4Q}|f(\xi)|\frac{1}{(t+|y-\xi|)^{n+\e}}d\xi\\
&\le& c_{n,\psi}(t/\ell_Q)^{\e}\sum_{k=0}^{\infty}\frac{1}{2^{k\e}}\frac{1}{|2^kQ|}\int_{2^kQ}|f|.
\end{eqnarray*}
Hence, using Chebyshev's inequality and that $\int_{{\mathbb R}^n}\Phi\Big(\frac{x-y}{t\a}\Big)dx\le c_n(t\a)^n$, we have
\begin{eqnarray*}
&&(I_1(f\chi_{{\mathbb R}^n\setminus 4Q})\chi_Q)^*(\la|Q|/2)\\
&&\le \frac{2}{\la|Q|}
\iint_{T(2Q)}\Big(\int_{{\mathbb R}^n}\Phi\Big(\frac{x-y}{t\a}\Big)dx\Big)|(f\chi_{{\mathbb R}^n\setminus 4Q})*\psi_t(y)|^2\frac{dydt}{t^{n+1}}\\
&&\le c_{n,\la,\psi}\a^n(1/\ell_Q)^{2\e}\left(\sum_{k=0}^{\infty}\frac{1}{2^{k\e}}\frac{1}{|2^kQ|}\int_{2^kQ}|f|\right)^2\int_0^{2\ell_Q}t^{2\e-1}dt\\
&&\le c_{n,\la,\psi}\a^n\left(\sum_{k=0}^{\infty}\frac{1}{2^{k\e}}\frac{1}{|2^kQ|}\int_{2^kQ}|f|\right)^2.
\end{eqnarray*}
By H\"older's inequality,
$$
\left(\sum_{k=0}^{\infty}\frac{1}{2^{k\e}}\frac{1}{|2^kQ|}\int_{2^kQ}|f|\right)^2\le
\left(\sum_{k=0}^{\infty}\frac{1}{2^{k\e}}\right)\sum_{k=0}^{\infty}\frac{1}{2^{k\e}}\left(\frac{1}{|2^kQ|}\int_{2^kQ}|f|\right)^2.
$$
Combining this with the previous estimate and with (\ref{weak11}) and (\ref{reares}) proves (\ref{1cl}).

Let $x,x_0\in Q$, and let us estimate now $|I_2(f)(x)-I_2(f)(x_0)|$. We have
\begin{eqnarray*}
&&|I_2(f)(x)-I_2(f)(x_0)|\\
&&\le \sum_{k=1}^{\infty}\iint_{T(2^{k+1}Q)\setminus T(2^kQ)}
\Big|\Phi\Big(\frac{x-y}{t\a}\Big)-\Phi\Big(\frac{x_0-y}{t\a}\Big)\Big||f*\psi_t(y)|^2\frac{dydt}{t^{n+1}}.
\end{eqnarray*}

Suppose $(y,t)\in T(2^{k+1}Q)\setminus T(2^kQ)$. If $y\in 2^kQ$, then $t\ge 2^k\ell_Q$. On the other hand, if
$y\in 2^{k+1}Q\setminus 2^kQ$, then for any $x\in Q$, $|y-x|\ge (2^k-1/2)\ell_Q$. Hence, if $t<\frac{1}{2\a}(2^k-1/2)\ell_Q$,
then $|y-x|/\a t>2$ and $|y-x_0|/\a t>2$, and therefore,
$$\Phi\Big(\frac{x-y}{t\a}\Big)-\Phi\Big(\frac{x_0-y}{t\a}\Big)=0.$$
Using also that
$$\Big|\Phi\Big(\frac{x-y}{t\a}\Big)-\Phi\Big(\frac{x_0-y}{t\a}\Big)\Big|\le \frac{\sqrt n\ell_Q}{\a t}\|\nabla \Phi\|_{L^{\infty}},$$
we get
\begin{eqnarray*}
&&\Big|\Phi\Big(\frac{x-y}{t\a}\Big)-\Phi\Big(\frac{x_0-y}{t\a}\Big)\Big|\chi_{\{T(2^{k+1}Q)\setminus T(2^kQ)\}}(y,t)\\
&&\le c_n\frac{\ell_Q}{\a t}\chi_{\{(y,t):y\in 2^{k+1}Q, 2^{k-2}\ell_Q/\a\le t\le 2^{k+1}\ell_Q}\}(y,t).
\end{eqnarray*}
Hence,
\begin{eqnarray*}
&&\iint_{T(2^{k+1}Q)\setminus T(2^kQ)}
\Big|\Phi\Big(\frac{x-y}{t\a}\Big)-\Phi\Big(\frac{x_0-y}{t\a}\Big)\Big||f*\psi_t(y)|^2\frac{dydt}{t^{n+1}}\\
&&\le c_n\frac{\ell_Q}{\a}\int_{2^{k-2}\ell_Q/\a}^{2^{k+1}\ell_Q}\int_{2^{k+1}Q}|f*\psi_t(y)|^2\frac{dydt}{t^{n+2}}\le c_n(J_1+J_2),
\end{eqnarray*}
where
$$
J_1=\frac{\ell_Q}{\a}\int_{2^{k-2}\ell_Q/\a}^{2^{k+1}\ell_Q}\int_{2^{k+1}Q}
|(f\chi_{2^{k+2}Q})*\psi_t(y)|^2\frac{dydt}{t^{n+2}}
$$
and
$$
J_2=\frac{\ell_Q}{\a}\int_{2^{k-2}\ell_Q/\a}^{2^{k+1}\ell_Q}\int_{2^{k+1}Q}
|(f\chi_{{\mathbb R}^n\setminus2^{k+2}Q})*\psi_t(y)|^2\frac{dydt}{t^{n+2}}.
$$

Let us first estimate $J_1$. Using Minkowski's integral inequality, we obtain
\begin{eqnarray*}
J_1\le \frac{\ell_Q}{\a}\left(\int_{2^{k+2}Q}|f(\xi)|\Big(
\int_{2^{k-2}\ell_Q/\a}^{2^{k+1}\ell_Q}\int_{2^{k+1}Q}\psi_t(y-\xi)^2\frac{dydt}{t^{n+2}}\Big)^{1/2}d\xi\right)^2.
\end{eqnarray*}
Since
$$
\int_{2^{k+1}Q}\psi_t(y-\xi)^2dy\le \frac{\|\psi\|_{L^{\infty}}}{t^n}\|\psi_t\|_{L^1}=
\frac{\|\psi\|_{L^{\infty}}\|\psi\|_{L^1}}{t^n},
$$
we get
\begin{eqnarray*}
J_1&\le& c_{\psi}\frac{\ell_Q}{\a}\Big(\int_{2^{k+2}Q}|f(\xi)|d\xi\Big)^2\int_{2^{k-2}\ell_Q/\a}^{\infty}\frac{dt}{t^{2n+2}}\\
&\le& c_{n,\psi}\a^{2n}2^{-k}\Big(\frac{1}{|2^{k+2}Q|}\int_{2^{k+2}Q}|f(\xi)|d\xi\Big)^2.
\end{eqnarray*}

We turn to the estimate of $J_2$. By (\ref{cond}), for $(y,t)\in T(2^{k+1}Q)$,
\begin{eqnarray*}
|(f\chi_{{\mathbb R}^n\setminus 2^{k+2}Q})*\psi_t(y)|&\le& c_{\psi}t^{\e}\int_{{\mathbb R}^n\setminus 2^{k+2}Q}|f(\xi)|\frac{1}{(t+|y-\xi|)^{n+\e}}d\xi\\
&\le& c_{n,\psi}(t/\ell_Q)^{\e}\sum_{i=k}^{\infty}\frac{1}{2^{i\e}}\frac{1}{|2^iQ|}\int_{2^iQ}|f|.
\end{eqnarray*}
Therefore,
\begin{eqnarray*}
J_2&\le& c_{n,\psi}\frac{\ell_Q}{\a}\Big(\sum_{i=k}^{\infty}\frac{1}{2^{i\e}}\frac{1}{|2^iQ|}\int_{2^iQ}|f|\Big)^2\frac{1}{\ell_Q^{2\e}}
\int_{2^{k-2}\ell_Q/\a}^{2^{k+1}\ell_Q}\int_{2^{k+1}Q}\frac{dydt}{t^{n+2-2\e}}\\
&\le& c_{n,\psi}\a^{n-2\e}2^{(2\e-1)k}\Big(\sum_{i=k}^{\infty}\frac{1}{2^{i\e}}\frac{1}{|2^iQ|}\int_{2^iQ}|f|\Big)^2.
\end{eqnarray*}

Combining the estimates for $J_1$ and $J_2$, we obtain
\begin{eqnarray*}
|I_2(f)(x)-I_2(f)(x_0)|&\le& c_{n,\psi}\a^{2n}\sum_{k=1}^{\infty}\frac{1}{2^k}\Big(\frac{1}{|2^{k}Q|}\int_{2^{k}Q}|f(\xi)|d\xi\Big)^2\\
&+& c_{n,\psi}\a^{n-2\e}\sum_{k=1}^{\infty}\frac{2^{2\e k}}{2^k}\Big(\sum_{i=k}^{\infty}\frac{1}{2^{i\e}}\frac{1}{|2^iQ|}\int_{2^iQ}|f|\Big)^2.
\end{eqnarray*}

By H\"older's inequality,
\begin{eqnarray*}
&&\sum_{k=1}^{\infty}\frac{2^{2\e k}}{2^k}\left(\sum_{i=k}^{\infty}\frac{1}{2^{i\e}}\frac{1}{|2^iQ|}\int_{2^iQ}|f|\right)^2\\
&&\le c_{\e}\sum_{k=1}^{\infty}\frac{2^{\e k}}{2^k}\sum_{i=k}^{\infty}\frac{1}{2^{i\e}}\left(\frac{1}{|2^iQ|}\int_{2^iQ}|f|\right)^2\\
&&\le c_{\e}\sum_{k=1}^{\infty}\ga(k,\e)\left(\frac{1}{|2^kQ|}\int_{2^kQ}|f|\right)^2,
\end{eqnarray*}
where
$$
\ga(k,\e)=\begin{cases} \frac{1}{2^{\e k}}, & \e<1\\
\frac{k}{2^k}, & \e=1.\end{cases}
$$
Therefore,
$$
|I_2(f)(x)-I_2(f)(x_0)|\le c_{n,\psi}\a^{2n}\sum_{k=1}^{\infty}\ga(k,\e)\left(\frac{1}{|2^kQ|}\int_{2^kQ}|f|\right)^2.
$$
From this and from (\ref{1cl}),
\begin{eqnarray*}
\o_{\la}(\widetilde S_{\a,\psi}(f)^2;Q)&\le& (I_1(f)\chi_Q)^*(\la|Q|)+\|I_2(f)-I_2(f)(x_0)\|_{L^{\infty}(Q)}\\
&\le& c_{n,\la,\psi}\a^{2n}\sum_{k=0}^{\infty}\ga(k,\e)\left(\frac{1}{|2^kQ|}\int_{2^kQ}|f|\right)^2,
\end{eqnarray*}
which completes the proof.
\end{proof}

\section{Proof of Theorem \ref{mainr}}
\subsection{Several auxiliary operators}
Given a sparse family ${\mathcal S}=\{Q_j^k\}\in {\mathscr{D}}$, define
$${\mathcal T}^{\mathcal S}_{2,m}f(x)=\left(\sum_{j,k}(f_{2^mQ_j^k})^2\chi_{Q_j^k}(x)\right)^{1/2}.$$

In the case when $m=0$, the following result was proved in \cite{CMP}.

\begin{lemma}\label{norms} For any $1<p<\infty$,
$$
\|{\mathcal T}^{\mathcal S}_{2,0}\|_{L^p(w)}\le c_{n,p}[w]_{A_p}^{\max(\frac{1}{2},\frac{1}{p-1})}.
$$
\end{lemma}

Given a sparse family ${\mathcal S}=\{Q_j^k\}\in {\mathcal D}$, define
$$
{\mathscr{M}}_{m}^{\mathcal S}(f,g)(x)=\sum_{j,k}(f_{2^mQ_{j,k}})\left(\frac{1}{|2^mQ_j^k|}\int_{Q_j^k}g\right)\chi_{2^mQ_j^k}(x).
$$

Applying Proposition \ref{prhp}, we decompose the cubes $Q_j^k$ into $2^n$ disjoint families $F_{i}$ such that for any $Q_j^k\in F_{i}$ there exists
a cube $P_{j,k}^{m,i}\in {\mathscr{D}}_{i}$ such that $2^mQ_j^k\subset P_{j,k}^{m,i}$ and
$\ell_{P_{j,k}^{m,i}}\le 6\ell_{2^mQ_j^k}$. Hence,
\begin{equation}\label{est}
{\mathscr{M}}_{m}^{\mathcal S}(f,g)(x)\le 6^{2n}\sum_{i=1}^{2^n} {\mathscr{M}}_{i,m}^{\mathcal S}(f,g)(x),
\end{equation}
where
$$
{\mathscr{M}}_{i,m}^{\mathcal S}(f,g)(x)
=\sum_{j,k}(f_{P_{j,k}^{m,i}})\left(\frac{1}{|P_{j,k}^{m,i}|}\int_{Q_j^k}g\right)\chi_{P_{j,k}^{m,i}}(x).
$$

The following statement was obtained in \cite{DLP}.

\begin{lemma}\label{mult}
Suppose that the sum defining ${\mathscr{M}}_{i,m}^{\mathcal S}(f,g)$ is finite. Then there are at most $2^n$ cubes $Q_{\nu}\in {\mathscr{D}}_{i}$
covering the support of ${\mathscr{M}}_{i,m}^{\mathcal S}(f,g)$ so that for any $Q_{\nu}$
there are two sparse families ${\mathcal S}_{i,1}$ and ${\mathcal S}_{i,2}$ from ${\mathscr{D}}_{i}$ such that
for a.e. $x\in Q_{\nu}$,
$${\mathscr{M}}_{i,m}^{\mathcal S}(f,g)(x)\le c_nm\sum_{\kappa=1}^2\sum_{Q_j^k\in {\mathcal S}_{i,\kappa}}
f_{Q_j^k}g_{Q_j^k}\chi_{Q_j^k}(x).$$
\end{lemma}

Observe that the proof of Lemma \ref{mult} is based on Theorem \ref{decom1} along with \cite[Lemma 3.2]{Ler}.
Formally Lemma \ref{mult} follows from \cite[Lemma 4.2]{DLP} taking there $m=2$ (which corresponds to a bilinear case) and $l=m$, and from the subsequent argument in \cite[Section 4.2]{DLP}.

Let $X$ be a Banach function space, and let $X'$ denote the associate space (see \cite[Ch. 1]{BS}).
Given a Banach function space $X$, denote by $X^{(2)}$ the space endowed with the norm
$$\|f\|_{X^{(2)}}=\||f|^2\|_{X}^{1/2}.$$ It is well known \cite[Ch. 1]{LT} that
$X^{(2)}$ is also a Banach space.

\begin{lemma}\label{estim} For any Banach function space $X$,
$$\sup_{{\mathcal S}\in {\mathcal D}}\|{\mathcal T}^{\mathcal S}_{2,m}f\|_{X^{(2)}}\le c_nm^{1/2}\max_{1\le i\le 2^n}\sup_{{\mathcal S}\in {\mathscr{D}}_{i}}\|{\mathcal T}^{\mathcal S}_{2,0}f\|_{X^{(2)}}.$$
\end{lemma}

\begin{proof}
By the standard argument, one can assume that the sum defining ${\mathcal T}^{\mathcal S}_{2,m}f$ is finite.
Fix ${\mathcal S}\in {\mathcal D}$.
By duality, there exists $g\ge 0$ with $\|g\|_{X'}=1$ such that
\begin{eqnarray}
\|{\mathcal T}^{\mathcal S}_{2,m}f\|_{X^{(2)}}^2&=&\int_{{\mathbb R}^n}
({\mathcal T}^{\mathcal S}_{2,m}f)^2g\,dx=\sum_{j,k}(f_{2^mQ_j^k})^2\int_{Q_j^k}g\label{imes}\\
&=&\int_{{\mathbb R}^n}{\mathscr{M}}_{m}^{\mathcal S}(f,g)f\,dx.\nonumber
\end{eqnarray}
Observe that the sum defining ${\mathscr{M}}_{m}^{\mathcal S}(f,g)$ is finite.
By Lemma \ref{mult} and by H\"older's inequality,
\begin{eqnarray*}
\int_{Q_{\nu}}{\mathscr{M}}_{i,m}^{\mathcal S}(f,g)f\,dx&\le& c_nm\sum_{\kappa=1}^2\sum_{Q_j^k\in {\mathcal S}_{i,\kappa}}
(f_{Q_j^k})^2\int_{Q_j^k}g\\
&\le& c_nm\sum_{\kappa=1}^2\int_{{\mathbb R}^n}
({\mathcal T}^{{\mathcal S}_{i,\kappa}}_{2,0}f)^2g\,dx\\
&\le& 2c_nm\sup_{{\mathcal S}\in {\mathscr{D}}_{i}}\|{\mathcal T}^{\mathcal S}_{2,0}f\|_{X^{(2)}}^2.
\end{eqnarray*}
Summing up over $Q_{\nu}$ and using (\ref{est}), we obtain
$$
\int_{{\mathbb R}^n}{\mathscr{M}}_{m}^{\mathcal S}(f,g)f\,dx\le c_nm\max_{1\le i\le 2^n}
\sup_{{\mathcal S}\in {\mathscr{D}}_{i}}\|{\mathcal T}^{\mathcal S}_{2,0}f\|_{X^{(2)}}^2.
$$
This along with (\ref{imes}) completes the proof.
\end{proof}

\subsection{Proof of Theorem \ref{mainr}}
Let $Q\in {\mathcal D}$. Applying Theorem \ref{decom1} along with Lemma \ref{oscg}, we get
that there exists a sparse family ${\mathcal S}=\{Q_j^k\}\in {\mathcal D}(Q)$ such that for a.e. $x\in Q$,
$$|\widetilde S_{\a,\psi}(f)(x)^2-m_{Q}(\widetilde S_{\a,\psi}(f)^2)|\le c_{n,\psi}\a^{2n}\Big(Mf(x)^2+\sum_{m=0}^{\infty}\frac{1}{2^{m\d}}\big({\mathcal T}^{\mathcal S}_{2,m}f(x)\big)^2\Big).$$
Hence,
\begin{equation}\label{laseq}
|\widetilde S_{\a,\psi}(f)^2-m_{Q}(\widetilde S_{\a,\psi}(f)^2)|^{1/2}\le c_{n,\psi}\a^n\big(Mf(x)+{\mathcal T}(f)(x)\big),
\end{equation}
where
$${\mathcal T}(f)(x)=\sum_{m=0}^{\infty}\frac{1}{2^{m\d/2}}{\mathcal T}^{\mathcal S}_{2,m}f(x).$$

Assuming, for instance, that $f\in L^1$, and using (\ref{pro1}) and (\ref{weak1}), we get
$$\lim_{|Q|\to \infty}m_{Q}(\widetilde S_{\a,\psi}(f)^2)=0.$$
Therefore, letting $Q$ to anyone of $2^n$ quadrants and
using Fatou's lemma, by (\ref{laseq}) we obtain
\begin{equation}\label{gxp1}
\|\widetilde S_{\a,\psi}(f)\|_{L^p(w)}\le c_{n,\psi}\a^n\big(\|Mf\|_{L^p(w)}+\|{\mathcal T}(f)\|_{L^p(w)}\big).
\end{equation}

Combining Lemma \ref{norms} and Lemma \ref{estim} with $X=L^{3/2}(w)$ yields
\begin{eqnarray*}
\|{\mathcal T}(f)\|_{L^3(w)}&\le& \sum_{m=0}^{\infty}\frac{1}{2^{m\d/2}}\|{\mathcal T}^{\mathcal S}_{2,m}f\|_{L^3(w)}\\
&\le& c_n\sum_{m=0}^{\infty}\frac{m^{1/2}}{2^{m\d/2}}
\max_{1\le i\le 2^n}\sup_{{\mathcal S}\in {\mathscr{D}}_{i}}\|{\mathcal T}^{\mathcal S}_{2,0}f\|_{L^3(w)}\\
&\le& c_{n,\d}[w]_{A_3}^{1/2}\|f\|_{L^3(w)}.
\end{eqnarray*}
Hence, by the sharp version of the Rubio de Francia extrapolation theorem (see \cite{DGPP} or \cite{D}),
\begin{equation}\label{esti1}
\|{\mathcal T}(f)\|_{L^p(w)}\le c_{n,p,\d}[w]_{A_p}^{\max(\frac{1}{2},\frac{1}{p-1})}\|f\|_{L^p(w)}\quad(1<p<\infty).
\end{equation}
Thus, applying this result along with Buckley's estimate $\|M\|_{L^p(w)}\le c_{n,p}[w]_{A_p}^{\frac{1}{p-1}}$ (see \cite{B}) and (\ref{gxp1}),
we get
$$
\|S_{\a,\psi}\|_{L^p(w)}\le \|\widetilde S_{\a,\psi}\|_{L^p(w)}\le c_{n,p,\psi}\a^n[w]_{A_p}^{\max(\frac{1}{2},\frac{1}{p-1})},
$$
and therefore, the proof is complete.

\section{Concluding remarks}
In a recent work \cite{LS}, the following weak type estimate was obtained for $G_{\b}(f)$ (and hence for $S_{\psi}(f)$):
if $1<p<3$, then
$$\|G_{\b}(f)\|_{L^{p,\infty}(w)}\lesssim [w]_{A_p}^{\max(\frac{1}{2},\frac{1}{p})}\Phi_p([w]_{A_p})\|f\|_{L^p(w)},$$
where $\Phi_p(t)=1$ if $1<p<2$ and $\Phi_p(t)=1+\log t$ if $p\ge 2$. The proof was based on the local mean oscillation decomposition
technique along with the estimate
\begin{equation}\label{estt}
\|{\mathcal T}^{\mathcal S}_{2,0}f\|_{L^{p,\infty}(w)}\lesssim [w]_{A_p}^{\max(\frac{1}{2},\frac{1}{p})}\Phi_p([w]_{A_p})\|f\|_{L^p(w)}.
\end{equation}

Since the space $L^{p,\infty}(w)$ is normable if $p>1$ (see, e.g., \cite[p. 220]{BS}), combining Lemma \ref{estim} with $X=L^{1+\e,\infty}(w), \e>0,$ and (\ref{estt})
yields for $2<p<3$ that
\begin{equation}\label{estt1}
\|{\mathcal T}f\|_{L^{p,\infty}(w)}\lesssim [w]_{A_p}^{\max(\frac{1}{2},\frac{1}{p})}\Phi_p([w]_{A_p})\|f\|_{L^p(w)}.
\end{equation}
Hence, exactly as above, by (\ref{laseq}) (and by the weak type estimate for $M$ proved in \cite{B}), we obtain
$$
\|S_{\a,\psi}(f)\|_{L^{p,\infty}(w)}\lesssim \a^n[w]_{A_p}^{\max(\frac{1}{2},\frac{1}{p})}\Phi_p([w]_{A_p})\|f\|_{L^p(w)}\quad(2<p<3).
$$

We emphasize that our approach does not allow to extend this estimate to $1<p\le 2$. This is clearly related to the same problem with (\ref{estt1}).
The limitation $2<p<3$ in (\ref{estt1}) is due to Lemma \ref{estim} where the condition that $X$ is a Banach function space was essential in the proof.
This raises a natural question whether Lemma \ref{estim} holds under the condition that $X$ is a quasi-Banach space. Observe that the same question can
be asked regarding a recent estimate related $X$-norms of Calder\'on-Zygmund and dyadic positive operators \cite{Ler1}.


\begin{thebibliography}{99}

\bibitem{A}
P. Auscher, {\it Change of angle in tent spaces},
C. R. Math. Acad. Sci. Paris {\bf 349} (2011), no. 5-6, 297-–301.

\bibitem{BS}
C. Bennett and R. Sharpley, {\it Interpolation of Operators},
Academic Press, New York, 1988.

\bibitem{B}
S.M. Buckley, {\it Estimates for operator norms on weighted spaces
and reverse Jensen inequalities}, Trans. Amer. Math. Soc., {\bf 340}
(1993), no. 1, 253--272.

\bibitem{CMP}
D. Cruz-Uribe, J.M. Martell and C. P\'erez, {\it Sharp weighted
estimates for classical operators}, Adv. Math., {\bf 229} (2012), no. 1, 408--441.

\bibitem{DLP}
W. Dami\'an, A.K. Lerner and C. P\'erez, {\it Sharp weighted bounds for multilinear maximal functions and
Calder\'on-Zygmund operators}, preprint. Available at http://arxiv.org/abs/1211.5115

\bibitem{DGPP}
O. Dragi\v{c}evi\'c, L. Grafakos, M.C. Pereyra and S. Petermichl,
{\it Extrapolation and sharp norm estimates for classical operators
on weighted Lebesgue spaces}, Publ. Math., {\bf 49} (2005), no. 1,
73--91.

\bibitem{D}
J. Duoandikoetxea, {\it Extrapolation of weights revisited: new
proofs and sharp bounds}, J. Funct. Anal., {\bf 260} (2011),
1886-–1901.

\bibitem{F}
C. Fefferman, {\it Inequalities for strongly singular convolution operators},
Acta Math. {\bf 124} (1970), 9--36.

\bibitem{GCR}
J. Garc\'ia-Cuerva and J.L. Rubio de Francia,
Weighted norm inequalities and related topics, North-Holland (1985).

\bibitem{HP}
T. Hyt\"onen and C. P\'erez, {\it Sharp weighted bounds involving ${A}_{\infty}$},
Analysis \& PDE, to appear.

\bibitem{LS}
M.T. Lacey and J. Scurry, {\it Weighted weak type estimates for square functions},
preprint. Available at http://arxiv.org/abs/1211.4219


\bibitem{L1}
A.K. Lerner, {\it  A pointwise estimate for the local sharp maximal function with applications to singular integrals},
Bull. London Math. Soc., {\bf 42} (2010), no. 5, 843--856.

\bibitem{Lerner}
A.K. Lerner, {\it Sharp weighted norm inequalities for
Littlewood-Paley operators and singular integrals}, Adv. Math., {\bf
226} (2011), 3912--3926.

\bibitem{Ler}
A.K. Lerner, {\it A simple proof of the {$A_2$} conjecture},
Int. Math. Res. Not. 2012, doi:10.1093/imrn/rns145.

\bibitem{Ler1}
A.K. Lerner, {\it On an estimate of Calder\'on-Zygmund operators by dyadic positive operators},
J. Anal. Math., to appear. Available at http://arxiv.org/abs/1202.1860

\bibitem{LT}
J. Lindenstrauss and L. Tzafriri, Classical Banach Spaces II, Springer-Verlag, Berlin, 1979.

\bibitem{MW}
B. Muckenhoupt and R.L. Wheeden,
{\it Norm inequalities for the Littlewood-Paley function $g_{\la}^*$},
Trans. Amer. Math. Soc. {\bf 191}\,(1974), 95--111.

\bibitem{T}
A. Torchinsky, Real-variable methods in harmonic analysis, Academic Press, 1986.

\bibitem{Wil}
J.M. Wilson, {\it The intrinsic square function}, Rev. Mat.
Iberoamericana,  {\bf 23} (2007), 771--791.


\end{thebibliography}
\end{document}